\documentclass[preprint,12pt]{elsarticle}
\usepackage{amsthm,amsmath,amssymb}
\usepackage{graphicx}
\usepackage{longtable,booktabs}
\usepackage[colorlinks=true,citecolor=black,linkcolor=black,urlcolor=blue]{hyperref}
\usepackage{mathrsfs}
\usepackage{epstopdf}
\usepackage{epsfig}

\theoremstyle{plain}
\newtheorem{thm}{Theorem}
\newtheorem{lem}[thm]{Lemma}
\newtheorem{cor}[thm]{Corollary}
\newtheorem{proposition}[thm]{Proposition}

\theoremstyle{definition}

\theoremstyle{remark}

\journal{Linear and Multilinear Algebra}
\begin{document}
\title{Enumeration of spanning trees of middle digraphs}
\author{Xuemei Chen\\
\small School of Mathematical Sciences\\[-0.8ex]
\small Xiamen University\\[-0.8ex]
\small P. R. China\\
Xian'an Jin\\
\small School of Mathematical Sciences\\[-0.8ex]
\small Xiamen University\\[-0.8ex]
\small P. R. China\\
Weigen Yan\footnote{Corresponding author.}\\
\small School of Sciences\\[-0.8ex]
\small Jimei University\\[-0.8ex]
\small P. R. China\\
\small{\tt Email: xuemeichen@stu.xmu.edu.cn; xajin@xmu.edu.cn; weigenyan@jmu.edu.cn}
}

\begin{abstract}
Let $D$ be a connected weighted digraph. The relation between the vertex weighted complexity (with a fixed root) of the line digraph of $D$ and the edge weighted complexity (with a fixed root) of $D$ has been given in (L. Levine, Sandpile groups and spanning trees of directed line graphs, J. Combin. Theory Ser. A 118 (2011) 350-364) and, independently, in (S. Sato, New proofs for Levine's theorems,  Linear Algebra Appl. 435 (2011) 943-952).
In this paper, we obtain a relation between the vertex weighted complexity of the middle digraph of $D$ and the edge weighted complexity of $D$. Particularly, when the weight of each arc and each vertex of $D$ is 1, the enumerative formula of spanning trees of the middle digraph of a general digraph is obtained.
\end{abstract}


\maketitle

\section{Introduction}
Graphs and digraphs considered in this paper are all simple and finite, if not specified. Let $D=(V(D),A(D),\chi)$ be a connected weighted digraph with vertex set $V(D)$, arc set $A(D)$ and weight function $\chi:V(D)\cup A(D)\rightarrow(0,\infty)$. Set $\chi_{i}=\chi_{v_{i}}=\chi(v_{i})$ for $v_{i}\in V(D)$, and $\chi_{ij}=\chi_e=\chi(e)$ for $e=(v_{i},v_{j})\in A(D)$.

For any arc $e=(u,v)\in A(D)$, we set $u=t(e)$, the tail of $e$,  and $v=h(e)$, the head of $e$. For any vertex $v\in V(D)$, we set
$$d_{+}(v)=|\{e\in A(D)| t(e)=v\}|, d_{-}(v)=|\{e\in A(D)| h(e)=v\}|.$$

A spanning tree of $D$ is a connected subdigraph containing all vertices of $D$, having no cycles, in which one vertex $u$ (the root) has outdegree 0 (i.e., $d_+(u)=0$), and every other vertex has outdegree 1. Let $\mathcal{T}(D)$ be the set of all spanning trees of $D$ and $\mathcal T(D,v)$ be the set of spanning trees of $D$ with root $v$. Define the {\it edge weighted complexity} and the {\it vertex weighted complexity} of the digraph $D$ as follows:
\begin{center}
$\kappa ^{edge}(D,\chi)=\sum\limits_{T\in \mathcal{T}(D) } \prod \limits_{e\in A(T)}\chi _{e}$,
\end{center}
\begin{center}
$\kappa ^{vertex}(D,\chi)=\sum\limits_{T\in \mathcal{T}(D)} \prod \limits_{e\in A(T)}\chi _{h(e)}$.
\end{center}

For a fixed vertex $v\in V(D)$, define two polynomials as follows:
\begin{center}
$\kappa ^{edge}(D,v,\chi)=\sum\limits_{T\in \mathcal{T}(D,v)} \prod \limits_{e\in A(T)}\chi _{e}$,
\end{center}
\begin{center}
$\kappa ^{vertex}(D,v,\chi)=\sum\limits_{T\in \mathcal{T}(D,v)} \prod \limits_{e\in A(T)}\chi _{h(e)}$.
\end{center}

Let $D$ be a weighted digraph. The {\it edge-weighted Laplacian matrix $\Delta^{edge}$} and {\it vertex-weighted Laplacian matrix $\Delta^{vertex}$} are defined as follows:
$$
\Delta^{edge}=
\begin{pmatrix}
 \sum\limits_{j\neq1}\chi_{1j}& -\chi_{12}  & \cdots   &-\chi_{1n}   \\
-\chi_{21} & \sum\limits_{j\neq2}\chi_{2j}  & \cdots   & -\chi_{2n}  \\
\vdots & \vdots  & \ddots   & \vdots  \\
-\chi_{n1} &-\chi_{n2}  & \cdots\  & \sum\limits_{j\neq n}\chi_{nj} \\
\end{pmatrix},
$$
$\Delta^{vertex}=(d_{uv})_{u,v\in V(D)}$, where
\[d_{uv}=\begin{cases}
\sum_{t(e)=u}\chi(h(e))&u=v,\\
-\chi(v)& (u,v)\in A(D),\\
0& (u,v)\not \in A(D).
\end{cases}\]

In case that $(v_i,v_j)$ is not an arc, $\chi_{ij}=0$.

The problem related to the enumeration of spanning trees is one of basic problems not only in the field of algebraic graph theory but also in electric circuit theory, which has been investigated for more than 160 years. There are many methods to enumerate spanning trees of graphs and digraphs. For example, we can use the {\it matrix tree theorem} (see \cite{NB,JU,RH}) to enumerate spanning trees of the general graphs and the general digraphs. For the weighted digraph $D$, we can use the following {\it generalized matrix tree theorem}:

\begin{thm}(Generalized Matrix Tree Theorem)\label{thm1.3}

Let $D$ be a finite digraph. Then
\begin{center}
$\kappa ^{edge}(D,v,\chi)=\det((\Delta ^{edge})_{vv})$,
\end{center}
\begin{center}
$\kappa ^{vertex}(D,v,\chi)=\det((\Delta ^{vertex})_{vv})$,
\end{center}
where $(\Delta ^{edge})_{vv}$ and $(\Delta ^{vertex})_{vv}$ are two matrices obtained from $\Delta ^{edge}$ and $\Delta ^{vertex}$ by deleting row $v$ and column $v$ of $\Delta ^{edge}$ and $\Delta ^{vertex}$, respectively.

Furthermore,
\begin{center}
$\kappa ^{edge}(D,\chi)=tr(adj \Delta ^{edge})=\sum\limits_{v\in V(D)}\kappa ^{edge}(D,v,\chi),$
\end{center}
\begin{center}
$\kappa ^{vertex}(D,\chi)=tr(adj \Delta ^{vertex})=\sum\limits_{v\in V(D)}\kappa ^{vertex}(D,v,\chi)$,
\end{center}
where $adj A$ is the cofactor matrix of a square matrix $A$.
\end{thm}
For the proof of the {\it generalized matrix tree theorem}, see for example Section 5 in \cite{JBO} for the edge-weighted version, and Theorems 1 and 2 in \cite{FR} for the vertex-weighted version.

The (vertex-)weighted line digraph of a weighted digraph $D$ with vertex set $V(D)$, arc set $A(D)$, and the weight function is $\chi$, denoted by $L(D)$, has vertex set $A(D)$, and for any two arcs $e_1=(u_1,v_1), e_2=(u_2,v_2)\in A(D)$, there exists an arc $(e_1,e_2)$ in $L(D)$ if and only if $v_1=u_2$, and the weight of each vertex $e$ of $L(D)$ equals the weight of the edge $e$ in $D$. For the weighted line digraph $L(D)$ of the digraph $D$, the relationship between the vertex weighted complexity of the weighted line graph $L(D)$ and the edge weighted complexity of the original digraph $D$ has been given by Levine \cite{LL}. Independently, Sato \cite{IS} presented a new proof for two Levine's Theorems by using the {\it generalized matrix tree theorem}.
\begin{thm}\cite{LL}
Let $D$ be a finite weighted digraph. Then
\begin{eqnarray}
\kappa^{vertex}(L(D),\chi)=\kappa^{edge}(D,\chi)\displaystyle\prod_{i = 1}^{n}d_{i}^{r_{i}-1}.\label{eq1}
\end{eqnarray}
where L(D) is the line digraph of $D$, $n=|V(D)|$, $m=|A(D)|$ and $r_{i}=d_{-}(v_{i})$, $d_{i}=\sum\limits_{t(e)=v_{i}}\chi_{e}$.
\end{thm}
\begin{thm}\cite{LL}
Let $D$ be a finite weighted digraph, and $e_{*}=(w_{*},v_{*})$ be an arc of $D$. Suppose $d_{-}(v)\geq 1$ for all vertices $v\in V(D)$. Then
\begin{eqnarray}\kappa^{vertex}(L(D),e_{*},\chi)=\chi_{e_{*}}\kappa^{edge}(D,w_{*},\chi)d_{v_{*}}^{r_{v_{*}}-2}\displaystyle\prod_{v\neq v_{*}}d_{v}^{r_{v}-1}.\label{eq2}
\end{eqnarray}
\end{thm}
For a weighted digraph $D$, the (vertex-)weighted middle digraph $M(D)$ of $D$ is the weighted digraph obtained from $D$ by replacing each arc $e=(u,v)$ with a directed path $u\rightarrow e\rightarrow v$, and adding a new arc from $e$ to $f$ if the arcs $e$ and $f$ in $D$ satisfy $h(e)=t(f)$. For each vertex of $M(D)$, if $v\in V(D)$, its weight is the weight of vertex $v$ of original digraph $D$, otherwise, its weight is the weight of corresponding arc of $D$.
We define two other sets as follows:
\begin{small}$$EE(D)=\{(e,f)\in A(D) \times A(D)\mid h(e)=t(f)\},$$
$$EV(D)=\{(v,e)\in V(D)\times A(D)\mid t(e)=v\}\cup\{(e,v)\in A(D)\times V(D)\mid h(e)=v\}.$$
\end{small}

Then we can write the weighted middle digraph $M(D)$ of $D$ as follows:
$$M(D)=(V(D)\cup A(D),EE(D)\cup EV(D),\chi).$$

The middle graph and the middle digraph have extensively been studied. See, for example, the enumerative problem of spanning trees of the middle graph of the semiregular bipartite graph \cite{ISZ}, the regular graph \cite{JS}, the general graph \cite{WY} and the weighted graph \cite{JC}. Zamfirescu \cite{CZ} obtained the local and global characterizations of middle digraphs; Liu, Zhang and Meng \cite{LZM} studied super-arc-connected and super-connected middle digraphs and the spectra of middle digraphs; In 2016, Zamfirescu \cite{CMDZ} proved that if a digraph $D$ contains no loops, the intersection number of $M(D)$ is equal to the number of vertices of $D$ that are not sources, added to the number of vertices of $D$ that are not sinks.

The main purpose of this paper is to consider the relation between the vertex weighted complexity of $M(D)$ and the edge weighted complexity of a digraph $D$. Particularly, we derive an enumerative formula on the number of spanning trees of the middle digraph $M(D)$ of a digraph $D$.

\section{Main results}
Let $f(A,\lambda)=\det(\lambda I_n-A)$ be the characteristic polynomial of a matrix $A$ of order $n$, where $I_n$ is the unit matrix of order $n$. Firstly, we express the characteristic polynomial of $\Delta^{vertex}(M(D))$ in terms of $\Delta^{edge}(D)$.
\begin{lem}\label{lem1.2}
Suppose $D$ is a weighted digraph and $M(D)$ is the middle digraph of $D$. 
Then the characteristic polynomial of $\Delta^{vertex}(M(D))$ can be expressed by
\begin{eqnarray}f(\Delta^{vertex}(M(D)),\lambda)=f(\Delta^{edge}(D),\lambda)\displaystyle\prod_{i = 1}^{n}(\lambda-\chi _{i}-d_{i})^{r_{i}},\label{eq4}\end{eqnarray}
where $n=|V(D)|$, $r_{i}=d_{-}(v_{i})$, $d_{i}=\displaystyle\sum\limits_{t(e)=v_{i}}\chi _{e}$ $(1\leq i\leq n)$.
\end{lem}
\begin{proof}
Let $D$ be a digraph with $n$ vertices and $m$ arcs, and a weight function $\chi:V(D)\cup A(D)\rightarrow(0,\infty)$.
Let $W=(a_{vw})_{n\times n}, F=(f_{vw})_{n\times n}, W^{\iota}=(c_{ef})_{m\times m}, F^{\iota}=(h_{ef})_{m\times m}, M=(M_{ve})_{n\times m}, L=(L_{ev})_{m\times n}, Q=(Q_{vw})_{n\times n}, B=(B_{ef})_{m\times m}$ which are defined as follows.
\begin{align*}
&W:a_{vw}=
\begin{cases}
\chi_{e} &  e=(v,w)\in A(D) \cr  0 &e=(v,w)\not \in A(D).\end{cases}
\hspace{0.4in}
&&F:f_{vw}=
\begin{cases}
 \sum\limits _{t(e)=v}\chi_{e}&  v=w\cr  0 &v\neq w.\end{cases}\\
&W^{\iota}:c_{ef}=
\begin{cases}
\chi_{f} &  h(e)=t(f)\cr  0 &h(e)\neq t(f).\end{cases}
\hspace{0.9in}
&&F^{\iota}:h_{ef}=
\begin{cases}
 \sum\limits _{h(e)=t(g)}\chi_{g}&  e=f\cr  0 &e\neq f.\end{cases}\\
&M:M_{ve}:=
\begin{cases}
 \chi_{e}&  t(e)=v\cr  0 &t(e)\neq v.\end{cases}
 \hspace{0.7in}
 &&L:L_{ev}:=
 \begin{cases}
 1&  h(e)=v\cr  0 &h(e)\neq v.\end{cases}\\
&Q: Q_{vw}:=
\begin{cases}
 \chi_{v}&  v=w\cr  0 &v\neq w.\end{cases}
 \hspace{0.9in}
 &&B: B_{ef}:=
\begin{cases}
 \chi_{h(e)}&  e=f\cr  0 &e\neq f.\end{cases}
\end{align*}

It is not difficult to prove that $\Delta^{edge}(D)=F-W$, $W^{\iota}=LM$, $W=ML$, and
\begin{center}
$\Delta ^{vertex}(M(D))=\left(
                          \begin{array}{cc}
                            F & -M \\
                            -LQ & B+F^{\iota}-W^{\iota} \\
                          \end{array}
                        \right).
$
\end{center}
Thus,\begin{eqnarray*}
& &f(\Delta^{vertex}(M(D)),\lambda)\\
&=& \det(\lambda I_{m+n}-\Delta ^{vertex}(M(D)))\\
&=&\det\left(
              \begin{array}{cc}
                \lambda I_{n}-F & M \\
                LQ & \lambda I_{m}-B-F^{\iota}+W^{\iota} \\
              \end{array}
            \right)\\\\
&=&\det\left(
            \begin{array}{cc}
               \lambda I_{n}-F & M \\
              0&  (\lambda I_{m}-B-F^{\iota}+W^{\iota})-LQ(\lambda I_{n}-F)^{-1}M \\
            \end{array}
          \right)\\\\
&=&\det(\lambda I_{n}-F) \det(\lambda I_{m}-B-F^{\iota}-L[Q(\lambda I_{n}-F)^{-1}-I_{n}]M)\\
&=&\det(\lambda I_{n}-F) \det(I_{m}-L[Q(\lambda I_{n}-F)^{-1}-I_{n}]\times\\
& &M(\lambda I_{m}-B-F^{\iota})^{-1})\det(\lambda I_{m}-B-F^{\iota}).
\end{eqnarray*}

If $X$ is an $m\times n$ matrix and $Y$ is an $n\times m$ matrix, then
$$\det(I_{m}-XY)=\det(I_{n}-YX).$$

Thus,\begin{eqnarray*}
&&\det(\lambda I_{m+n}-\Delta ^{vertex}(M(D)))\\
&=& \det(\lambda I_{n}-F)\det(\lambda I_{m}-B-F^{\iota})\times\\
& &\det(I_{n}-[Q(\lambda I_{n}-F)^{-1}-I_{n}]M(\lambda I_{m}-B-F^{\iota})^{-1}L).
\end{eqnarray*}

Obviously,  $\det(\lambda I_{m}-B-F^{\iota})=\displaystyle\prod_{i = 1}^{n}(\lambda-\chi _{i}-d_{i})^{r_{i}}$.

By a suitable labelling vertices of $L(D)$, we have
%

$$(\lambda I_{m}-B-F^{\iota})^{-1}\\
=\left(
                                    \begin{array}{cccc}
                                    \begin{smallmatrix}
                                      (\lambda-\chi_{1}-d_{1})^{-1} I_{r_{1}}&   &   &   \\
                                       &  (\lambda-\chi_{2}-d_{2})^{-1} I_{r_{2}}&  &  \\
                                       &  &\ddots  &  \\
                                       &  &  &  (\lambda-\chi_{n}-d_{n})^{-1} I_{r_{n}}\\
\end{smallmatrix}
                                    \end{array}
                                  \right).
$$

For each arc $e=(v,w)\in A(D)$,

$$(M(\lambda I_{m}-B-F^{\iota})^{-1}L)_{vw}=\chi_{vw}\cdot(\lambda-\chi_{w}-d_{w})^{-1},$$
and
%
%
$$Q(\lambda I_{n}-F)^{-1}-I_{n}
 =\left(
                                 \begin{array}{cccc}
  \begin{smallmatrix}
                                   \chi_{1}(\lambda-d_{1})^{-1}-1 &  &  &  \\
                                    & \chi_{2}(\lambda-d_{2})^{-1}-1&  &  \\
                                    &  & \ddots   &  \\
                                    &  &  & \chi_{n}(\lambda-d_{n})^{-1}-1 \\
  \end{smallmatrix}
                                 \end{array}
                               \right).
$$
Thus,
\begin{eqnarray*}
& &\det(I_{n}-[Q(\lambda I_{n}-F)^{-1}-I_{n}]M(\lambda I_{m}-B-F^{\iota})^{-1}L)\\
&&\\
&=&\det\left(
       \begin{array}{cccc}
         1 & \frac{-[\chi_{1}(\lambda-d_{1})^{-1}-1]\cdot\chi_{12}}{\lambda-\chi_{2}-d_{2}} &\cdots &\frac{-[\chi_{1}(\lambda-d_{1})^{-1}-1]\cdot\chi_{1n}}{\lambda-\chi_{n}-d_{n}} \\
         \frac{-[\chi_{2}(\lambda-d_{2})^{-1}-1]\cdot\chi_{21}}{\lambda-\chi_{1}-d_{1}}& 1 & \cdots  &  \frac{-[\chi_{2}(\lambda-d_{2})^{-1}-1]\cdot\chi_{2n}}{\lambda-\chi_{n}-d_{n}}  \\
          \vdots &  \vdots  & \ddots  & \vdots   \\
          \frac{-[\chi_{n}(\lambda-d_{n})^{-1}-1]\cdot\chi_{n1}}{\lambda-\chi_{1}-d_{1}} &  \frac{-[\chi_{n}(\lambda-d_{n})^{-1}-1]\cdot\chi_{n2}}{\lambda-\chi_{2}-d_{2}} & \cdots &  1\\
       \end{array}
     \right)\\
&&\\
&=&\displaystyle\prod_{i = 1}^{n}(\lambda-\chi _{i}-d_{i})^{-1}[\chi_{i}(\lambda-d_{i})^{-1}-1](-1)^{n}\det(\lambda I_{n}-F+W).
\end{eqnarray*}
Therefore,
$$\det(\lambda I_{m+n}-\Delta ^{vertex}(M(D))=\displaystyle\prod_{i = 1}^{n}(\lambda-\chi _{i}-d_{i})^{r_{i}}\det(\lambda I_{n}-\Delta^{edge}(D)).$$
\end{proof}
By Lemma \ref{lem1.2}, we can obtain the relation between the complexities of $M(D)$ and $D$ as follows.
\begin{thm}\label{thm1.1}
Suppose $D$ is a weighted digraph and $M(D)$ is the middle digraph of $D$. Then,
\begin{eqnarray}\kappa^{vertex}(M(D),\chi)=\kappa^{edge}(D,\chi)\displaystyle\prod_{i = 1}^{n}(\chi _{i}+d_{i})^{r_{i}},\label{eq3}\end{eqnarray}
where $n=|V(D)|$, $m=|A(D)|$, $r_{i}=d_{-}(v_{i})$, $d_{i}=\displaystyle\sum\limits_{t(e)=v_{i}}\chi _{e}$ $(1\leq i\leq n)$.
\end{thm}

\begin{proof}
For the {\it characteristic polynomial} $f(\lambda)=\mid\lambda I_{n}-A\mid=\lambda^{n}+a_{n-1}\lambda^{n-1}+......+a_{1}\lambda+a_{0}$, the coefficient of $\lambda$ is $(-1)^{n-1}\times tr(adj A)$.

Therefore, by Theorem \ref{thm1.3}, we have
$$f^{'}(\Delta^{vertex}(M(D)),0)=(-1)^{m+n-1}\kappa^{vertex}(M(D),\chi),$$
$$f^{'}(\Delta^{edge}(D),0)=(-1)^{n-1}\kappa^{edge}(D,\chi).$$

By Lemma 4 (i.e., Eq. (3)),
$$f^{'}(\Delta^{vertex}(M(D)),0)=f^{'}(\Delta^{edge}(D),0)\displaystyle\prod_{i = 1}^{n}(\chi _{i}+d_{i})^{r_{i}}(-1)^{m}.$$

Therefore,
$$\kappa^{vertex}(M(D),\chi)=\kappa^{edge}(D,\chi)\displaystyle\prod_{i = 1}^{n}(\chi _{i}+d_{i})^{r_{i}}.$$
\end{proof}

When the weight of each arc and each vertex of $D$ is 1, the enumerative formula of spanning trees of the middle digraph of a digraph can be obtained as follows.
\begin{cor}
Suppose $D$ is a  digraph and $M(D)$ is the middle digraph of $D$. Then,
\begin{eqnarray}t(M(D))=t(D)\displaystyle\prod_{i = 1}^{n}(1+d_{i})^{r_{i}}.\label{eq5}\end{eqnarray}
where $t(M(D))$ is the number of the spanning trees of $M(D)$, $d_{i}=d_{+}(v_{i})$, $r_{i}=d_{-}(v_{i})$.
\end{cor}
If the edge $e_{*}=(w_{*},v_{*})$ is fixed, then the relation between the {\it vertex weighted complexity} of the middle digraph $M(D)$ with root $e_{*}$ and the {\it edge weighted complexity} of $D$ with root $w_{*}$ can be obtained as follows:
\begin{proposition}
Suppose $D$ is a weighted digraph and $M(D)$ is the middle digraph of $D$. For any fixed arc $e_{*}=(w_{*},v_{*})$ of $D$, then
\begin{eqnarray}\kappa^{vertex}(M(D),e_{*},\chi)=\chi_{e_{*}}\kappa^{edge}(D,w_{*},\chi)(\chi _{v_{*}}+d_{v_{*}})^{r_{v_{*}}-1}\displaystyle\prod_{v\neq v_{*}}(\chi _{v}+d_{v})^{r_{v}}.\label{eq6}\end{eqnarray}
\end{proposition}
\begin{proof}
Let $W=(a_{vw})_{n\times n}, F=(f_{vw})_{n\times n}, W^{\iota}=(c_{ef})_{m\times m}, F^{\iota}=(h_{ef})_{m\times m}, M=(M_{ve})_{n\times m}, L=(L_{ev})_{m\times n}, Q=(Q_{vw})_{n\times n}, B=(B_{ef})_{m\times m}$ be the eight matrices defined in the proof of Lemma 4.
Suppose that $e_{*}=(w_{*},v_{*})$ is a fixed arc in $D$.
Let $M_{0}, L_0, B_0, W_0^{\iota}, F_0^{\iota}$ be the matrices obtained from the matrix $M, L, B, W^{\iota}$ and $F_{\iota}$ by deleting the column $e_{*}$ of $M$, the row $e_{*}$ of $L$, the row $e_{*}$ and the column $e_{*}$ of $B$, $W^{\iota}$, $F^{\iota}$, respectively.
 Then, by
Theorem \ref{thm1.3} , we have
\begin{eqnarray*}
& &\kappa^{vertex}(M(D),e_{*},\chi)\\
  &=& \det\left(
                                            \begin{array}{cc}
                                              F & -M_{0} \\
                                              -L_{0}Q & B_{0}+D_{0}^{\iota}-W_{0}^{\iota} \\
                                            \end{array}
                                          \right)
  \\
   &=&\det\left(
                                            \begin{array}{cc}
                                              F & -M_{0} \\
                                              0 & B_{0}+F_{0}^{\iota}-W_{0}^{\iota}-L_{0}QF^{-1}M_{0} \\
                                            \end{array}
                                          \right) \\
  &=&\det(F)\det(B_{0}+F_{0}^{\iota}-W_{0}^{\iota}-L_{0}QF^{-1}M_{0})
   \\
   &=&\det(F)\det(I_{n}-(I_{n}+QF^{-1})M_{0}(B_{0}+F_{0}^{\iota})^{-1}L_{0})\det(B_{0}+F_{0}^{\iota}).
\end{eqnarray*}
But, we have $\det(B_{0}+F_{0}^{\iota})=(\chi _{v_{*}}+d_{v_{*}})^{r_{v_{*}}-1}\displaystyle\prod_{v\neq v_{*}}(\chi _{v}+d_{v})^{r_{v}}$.

Furthermore, we have
$$(B_{0}+F_{0}^{\iota})^{-1}=\left(
                              \begin{array}{ccccc}
                              \begin{smallmatrix}
                                (\chi_{1}+d_{1})^{-1}\cdot I_{r_{1}} &  &  &  &  \\
                                 & \ddots&  &  &  \\
                                        &  & (\chi_{v_{*}}+d_{v_{*}})^{-1}\cdot I_{r_{v_{*}-1}}  &  &  \\
                                 &  &  & \ddots &  \\
                                &  &  &  & (\chi_{n}+d_{n})^{-1}\cdot I_{r_{n}}  \\
                                \end{smallmatrix}
                              \end{array}
                            \right)
.$$
Thus, for an arc $e=(v, w)\in A(D)$,
$$(M_{0}(B_{0}+F_{0}^{\iota})^{-1}L_{0})_{wv}=\chi_{vw}\cdot(\chi_{w}+d_{w})^{-1}.$$
Therefore, it follows that
\begin{eqnarray*}
&&\det(I_{n}-(I_{n}+QF^{-1})M_{0}(B_{0}+F_{0}^{\iota})^{-1}L_{0})\\
&&\\
&=&\displaystyle\prod_{i = 1}^{n}(1+\chi_{i}d_{i}^{-1})\frac{1}{(\chi_{i}+d_{i})}\det\left(
                                                           \begin{array}{cccccc}
                                                           \begin{smallmatrix}
                                                              d_{1}& -\chi_{12} & \cdots&  -\chi_{1v_{*}}& \cdots & -\chi_{1n} \\
                                                               -\chi_{21} & d_{2} & \cdots &-\chi_{2v_{*}}& \cdots & -\chi_{2n} \\
                                                             \vdots& \vdots&  & \vdots &  & \vdots\\
                                                               -\chi_{w_{*}1}&  -\chi_{w_{*}2} & \cdots& 0 & \cdots &  -\chi_{w_{*}n}\\
                                                             \vdots & \vdots &  & \vdots &  & \vdots \\
                                                             -\chi_{n1} & -\chi_{n2} & \cdots & -\chi_{nv_{*}} & \cdots & d_{n} \\
                                                             \end{smallmatrix}
                                                           \end{array}
                                                         \right).
\end{eqnarray*}
Let $g_{w_{*}}=d_{_{w_{*}}}-\chi_{e_{*}}$, then
\begin{eqnarray*}
&=&\displaystyle\prod_{i = 1}^{n}(1+\chi_{i}d_{i}^{-1})\frac{1}{(\chi_{i}+d_{i})}\det\left(
         \begin{array}{cccccccc}
         \begin{smallmatrix}
           d_{1} & -\chi_{12} & \cdots& -\chi_{1w_{*}} & \cdots& -\chi_{1v_{*}} & \cdots & -\chi_{1n}\\
            -\chi_{21}&d_{2}&\cdots &-\chi_{2w_{*}}& \cdots & -\chi_{2v_{*}} & \cdots & -\chi_{2n} \\
           \vdots & \vdots &  & \vdots &  & \vdots & & \vdots \\
          -\chi_{w_{*}1} &  -\chi_{w_{*}2}& \cdots&g_{w_{*}}& \cdots & 0 &\cdots & -\chi_{nw_{*}} \\
          \vdots &\vdots&  & \vdots &  &  \vdots &  & \vdots \\
           -\chi_{n1} & -\chi_{n2} &\cdots & -\chi_{nw_{*}} & \cdots & -\chi_{nv_{*}} &\cdots& d_{n} \\
           \end{smallmatrix}
         \end{array}
       \right)\\
       &&\\
& &+\displaystyle\prod_{i = 1}^{n}(1+\chi_{i}d_{i}^{-1})\frac{1}{(\chi_{i}+d_{i})}\det\left(
         \begin{array}{cccccccc}
         \begin{smallmatrix}
           d_{1} & -\chi_{12} & \cdots& -\chi_{1w_{*}} & \cdots& -\chi_{1v_{*}} & \cdots & -\chi_{1n}\\
            -\chi_{12}&d_{2}&\cdots &-\chi_{2w_{*}}& \cdots & -\chi_{2v_{*}} & \cdots & -\chi_{2n} \\
           \vdots & \vdots &  & \vdots &  & \vdots & & \vdots \\
          0 &  0 & \cdots &\chi_{e_{*}}& \cdots & 0 &\cdots & 0 \\
          \vdots &\vdots&  & \vdots &  &  \vdots &  & \vdots \\
           -\chi_{n1} & -\chi_{2n} &\cdots & -\chi_{nw_{*}} & \cdots & -\chi_{nv_{*}} &\cdots& d_{n} \\
           \end{smallmatrix}
         \end{array}
       \right)\\
       &&\\
&=&\displaystyle\prod_{i = 1}^{n}(1+\chi_{i}d_{i}^{-1})\frac{1}{(\chi_{i}+d_{i})}\chi_{e_{*}}\kappa^{edge}(D,w_{*},\chi).
\end{eqnarray*}
Therefore,
\begin{eqnarray*}
& &\kappa^{vertex}(M(D),e_{*},\chi)\\
  &=& \det(F)\det(B_{0}+F_{0}^{\iota})\displaystyle\prod_{i = 1}^{n}(1+\chi_{i}d_{i}^{-1})\frac{1}{(\chi_{i}+d_{i})}\chi_{e_{*}}\kappa^{edge}(D,w_{*},\chi) \\
  &=& \displaystyle\prod_{i = 1}^{n}d_{i} (1+\chi_{i}d_{i}^{-1})\frac{1}{(\chi_{i}+d_{i})}(\chi _{v_{*}}+d_{v_{*}})^{r_{v_{*}}-1}\displaystyle\prod_{v\neq v_{*}}(\chi _{v}+d_{v})^{r_{v}}\chi_{e_{*}}\kappa^{edge}(D,w_{*},\chi)\\
  &=&\chi_{e_{*}}\kappa^{edge}(D,w_{*},\chi)(\chi _{v_{*}}+d_{v_{*}})^{r_{v_{*}}-1}\displaystyle\prod_{v\neq v_{*}}(\chi _{v}+d_{v})^{r_{v}}.
\end{eqnarray*}
\end{proof}

\section*{Acknowledgements}
\noindent

This work is supported by NSFC (Nos. 11671336, 12071180) and the Fundamental Research
Funds for the Central Universities  (No. 20720190062).

\section*{References}

\bibliographystyle{model1b-num-names}
\bibliography{<your-bib-database>}

\begin{thebibliography}{99}

\bibitem{NB}
N. Biggs, Algebraic Graph Theory, Cambridge University Press, Cambridge, UK, 1974.

\bibitem{JU}
J. A. Bondy, U. S. R. Murty, Graph Theory,  Graduate Texts in Mathematics, Springer, New York, 2008.

\bibitem{RH}
R. A. Brualdi, H. J. Ryser, Combinatorial Matrix Theory, Cambridge University Press, Cambridge, UK, 1991.

\bibitem{FR}
F. R. K. Chung, R. P. Langlands, A combinatorial Laplacian with vertex weights, J. Combin. Theory Ser. A  \textbf{75} (1996) 316-327.

\bibitem{JS}
J. Huang, S. Li, On the normalised Laplacian spectrum, degree-Kirchhoff index and spanning trees of graphs, Bull. Aust. Math. Soc. \textbf{91} (2015) 353-367.

\bibitem{LL}
L. Levine, Sandpile groups and spanning trees of directed line graphs, J. Combin. Theory Ser. A \textbf{118} (2011) 350-364.

\bibitem{LZM}
J. Liu, X. D. Zhang and J. X. Meng, The properties of middle digraphs, Ars Combin.  \textbf{101} (2011) 153-160.

\bibitem{JBO}
J. B. Orlin, Line-digraphs, arborescences, and theorems of Tutte and Knuth, J. Combin. Theory Ser. B  \textbf{25} (1978) 187-198.


\bibitem{IS}
I. Sato, New proofs for Levine's theorems,  Linear Algebra Appl. \textbf{435} (2011) 943-952.



\bibitem{ISZ}
I. Sato, Zeta functions and complexities of middle graphs of semiregular bipartite graphs, Discrete Math. \textbf{335} (2014) 92-99.


\bibitem{WY}
W. Yan, Enumeration of spanning trees of middle graphs, Appl. Math. Comput. \textbf{307} (2017) 239-243.


\bibitem{CZ}
C. Zamfirescu, Local and global characterizations of middle digraphs, in: The Theory and Applications of Graphs (G. Chartrand, Y. Alavi, D.L. Goldsmith, L. Lesniak-Foster and D.R. Lick, Eds.) Wiley, New York (1981) 593-607.


\bibitem{CMDZ}
CMD Zamfirescu, Transformations of digraphs viewed as intersection digraphs, Springer Proc. Math. Stat. \textbf{148} (2016) 27-35.

\bibitem{JC}
J. Zhou, C. Bu, The enumeration of spanning tree of weighted graphs,  J. Algebraic Combin.  \textbf{53} (2020) 1-34.

\end{thebibliography}

\end{document}